\documentclass[sn-mathphys]{sn-jnl}

\jyear{2021}%

\theoremstyle{thmstyleone}%
\newtheorem{theorem}{Theorem}[section]
\newtheorem{proposition}[theorem]{Proposition}%
\newtheorem{lemma}[theorem]{Lemma}%
\newtheorem{corollary}[theorem]{Corollary}%
\newtheorem{example}[theorem]{Example}%
\newtheorem{remark}[theorem]{Remark}

\theoremstyle{thmstylethree}%
\newtheorem{definition}[theorem]{Definition}%
\usepackage{mathrsfs}
\usepackage{stmaryrd}
\usepackage{amsmath}
\usepackage{amsthm}
\usepackage{enumitem}
\begin{document}

\title[On the $(b, c)$-inverse of a sum with a radical element in a ring]{On the $(b, c)$-inverse of a sum with a radical element in a ring}


\author[1]{\fnm{Yukun} \sur{Zhou}}\email{2516856280@qq.com}

\author*[2]{\fnm{N\'{e}stor} \sur{ Thome}}\email{njthome@mat.upv.es}


\affil[1]{\orgdiv{School of Physical and Mathematical Sciences}, \orgname{Nanjing Tech University}, \orgaddress{\city{Nanjing}, \postcode{211816}, \country{China}}}

\affil[2]{\orgdiv{Instituto Universitario de Matem\'{a}tica Multidisciplinar}, \orgname{Universitat Polit\`{e}cnica de Val\`{e}ncia}, \orgaddress{\city{Valencia}, \postcode{ 46022}, \country{Spain}}}


\abstract{Let $R$ be a ring with identity and $J(R)$ be its Jacobson radical. Assume that $a\in R$ is $(b,c)$-invertible and $j_a,j_b,j_c\in J(R)$. This paper provides necessary and sufficient conditions for $a+j_a$ to be $(b+j_b,c+j_c)$-invertible. As an
	application, corresponding results on $(\widehat{B},\widehat{C})$-inverses of a dual matrix $\widehat{A}$ are derived.}





\maketitle
\section { \bf Introduction}
Throughout the paper, let $R$ be a ring with identity and $J(R)$ be its Jacobson radical. As it is well known, if $a\in R$ is invertible and $j\in J(R)$, then $a+j$ is also invertible with $$(a+j)^{-1}=(1+a^{-1}j)^{-1}a^{-1}.$$ In 1986, Huylebrouck and Puystjens \cite{HuyPu} extended this result to several types of generalized inverses. They established necessary and sufficient conditions under which $a+j$ remains regular, Moore-Penrose invertible, or group invertible, provided that $a$ itself is, respectively, regular, Moore-Penrose invertible, or group invertible, and $j\in J(R)$. For further related results, we refer the reader  to \cite{3Hartwig,LCWX,Huy,YCh}. In particular, the following theorem  gives a key result concerning the Moore-Penrose inverse; herein, the Moore-Penrose inverse of $a$ is denoted by $a^{\dagger}$, and it is the only element that satisfies the four classical Moore-Penrose equations (see Definition \ref{bc2.1} below).

\begin{theorem}\label{bc1.1}\emph{\cite{HuyPu}} If $a\in R$ is Moore-Penrose invertible and $j\in J(R)$, then the following statements are equivalent:
	\begin{itemize}
		\item [{\rm (1)}] $a+j$ is Moore-Penrose invertible;
		\item [{\rm (2)}] $(1-aa^{\dagger})j(1+ a^{\dagger}j)^{-1}(1-a^{\dagger}a)=0$.
	\end{itemize}
\end{theorem}

Now, our attention shifts to dual matrices and their generalized inverses. A dual number $\widehat{a}$ \cite{Clifford1871} is defined as an expression of the form $$\widehat{a}=a+\varepsilon a_0,$$ where $a,a_0\in \mathbb{R}$ and $\varepsilon$  is the dual unit satisfying  $$\varepsilon\neq 0,~0\varepsilon=\varepsilon 0=0,~1\varepsilon=\varepsilon 1=\varepsilon,~\varepsilon^2=0.$$ The set of all dual numbers is denoted by $\mathbb{D}$, and it  forms a ring under addition and multiplication. A dual matrix $\widehat{A}\in \mathbb{D}^{m\times n}$ can be written as $\widehat{A}=A+\varepsilon A_0$, where $A$ and $A_0$ are real matrices. The transpose of $\widehat{A}$ is defined as $\widehat{A}^{\mathrm T}=A^{\mathrm T} +\varepsilon A_0^{\mathrm T}$.

The theory of dual matrices plays  a significant role in various fields, including kinematic analysis and synthesis,  robotics (cf. \cite{Pennestri2016,Brodsky1998,GuLuh1987,Hei1986}). In particular, generalized inverses of dual matrices are widely used in the study of kinematic problems and in solving linear dual equations (cf. \cite{Falco2018,UdMMT2021,WangGao2023,Wanghongxing2021}). In 2020, Udwadia et al. \cite{UdPen2020} pointed out that dual matrices may not always possess Moore-Penrose inverses. Subsequently, Wang \cite{Wanghongxing2021} established a necessary and sufficient condition which ensures the Moore-Penrose invertibility of a dual matrix, as stated below.

\begin{theorem}\label{bc1.2}\emph{\cite{Wanghongxing2021}} Let $\widehat{A}=A+\varepsilon A_0\in \mathbb{D}^{m\times n}$. The following statements are equivalent:
	\begin{itemize}
		\item [{\rm (1)}] $\widehat{A}$ is Moore-Penrose invertible;
		\item [{\rm (2)}] $(I-AA^{\dagger})A_0(I-A^{\dagger}A)=0$.
	\end{itemize}
\end{theorem}
Noting that $\mathbb{D}^{n\times n}$ also possesses a ring structure and that  $\varepsilon A_0$ lies in the Jacobson radical of $\mathbb{D}^{n\times n}$,  it is worth highlighting the interesting remark that Theorem  \ref{bc1.1} provides an explanation for Theorem  \ref{bc1.2}. Indeed, if $\widehat{A}$ is not square, it can be made square   by adding  adequately zero columns or zero rows. Furthermore, it is  straightforward to verify $(I-AA^{\dagger})(\varepsilon A_0)(I+ \widehat{A}^{\dagger}(\varepsilon A_0))^{-1}(I-A^{\dagger}A)=\varepsilon(I-AA^{\dagger})A_0(I-A^{\dagger}A)$.

The present article focuses on the $(b,c)$-inverse of a sum with Jacobson radical element. The concept of the $(b,c)$-inverse was introduced by the  distinguished  algebraist M.P. Drazin \cite{Drazin2012}, and serves as a unifying framework for many known generalized inverses,  including the Moore-Penrose inverse, the group inverse, the core inverse \cite{BT},  and the Drazin inverse \cite{D}. Let $R$ be  a ring with involution $\ast$, where an involution $a\mapsto a^*$ is an anti-isomorphism satisfying $(a+b)^{*}=a^{*}+b^{*}, (ab)^{*}=b^{*}a^{*}, (a^{*})^{*}=a$ for any $a, b\in R.$
For an element $a \in R$, the following correspondences hold:
\begin{itemize}
	\item [{\rm (1)}] the Moore-Penrose inverse of $a$ is its $(a^*,a^*)$-inverse;
	\item [{\rm (2)}] the group inverse of $a$ is its $(a,a)$-inverse;
	\item [{\rm (3)}] the core inverse of $a$ is its $(a,a^*)$-inverse;
	\item [{\rm (4)}] the Drazin inverse of $a$ is its $(a^k,a^k)$-inverse for some nonnegative integer $k$.
\end{itemize}
For more results of the $(b,c)$-inverse,  we refer the reader to  \cite{Dra2,2016kewangchen,Mary2012,Mosic2018,ZhuHuiHui2023,ZhuWu2021}.

The present paper first establishes necessary and sufficient conditions for the $(b,c)$-invertibility of a sum with a radical element. Subsequently, the expressions and idempotence of  generalized inverses of a sum with a radical element are investigated. Finally, the obtained results are applied to the case of dual matrices.

\section {Preliminaries}
In this section, we present some  definitions and notations.

\begin{definition}\cite{P}\label{bc2.1} Let $R$ be a ring with involution $\ast$ and $a\in R$. An element $x\in R$ that satisfies the following four equations:
	$${\rm (1)} ~axa=a,~~{\rm (2)}~ xax=x,~~{\rm (3)}~ (ax)^{\ast}=ax,~~{\rm (4)}~(xa)^{\ast}=xa,$$
	is called the Moore-Penrose inverse of $a$ and $a$ is said to be Moore-Penrose invertible. Such an $x$  is unique and denoted by $a^{\dagger}$.
\end{definition}
The element $x$ is called a $\{1\}$-inverse of $a$ (or $a$ is regular) if the equation ${\rm (1)}$ holds.  The element $x$ is called a $\{2\}$-inverse of $a$ if the equation ${\rm (2)}$ holds. The element $x$ is called a $\{1,2\}$-inverse (or a reflexive inverse) of $a$ if $x$ satisfies equations ${\rm (1)}$ and ${\rm (2)}$. The symbols $a\{1\}$, $a\{2\}$ and $a\{1,2\}$ denote the sets of all $\{1\}$-inverses,  $\{2\}$-inverses and $\{1,2\}$-inverses of $a$, respectively.

\begin{definition}\cite{D} Let $a\in R$. An element $x\in R$ that satisfies
	$$xax=x,~xa=ax,~xa^{k+1}=a^{k}~\text{for~some~nonnegative~integer}~k,$$
	is called the Drazin inverse of $a$, and  $a$ is said to be  Drazin invertible. Such an $x$ is unique and denoted by $a^{D}$. In particular, when $k=1$, $x$ is called the group inverse of $a$ and is denoted by $a^{\#}$.
\end{definition}
If $k$ is the smallest nonnegative integer such that  above equations hold, then $k$ is called the Drazin index of $a$ and denoted by ${\rm i}(a).$

\begin{definition}\cite{BT,RDD,XCZ} Let $a\in R$. An element $x\in R$   that satisfies
	$$xa^2=a,~~~ax^2=x, ~~~(ax)^{*}=ax,$$
	is called the core inverse of $a$, and $a$ is said to be core invertible. Such an $x$ is unique and denoted by $a^{\scriptsize\textcircled{\tiny \#}}$.
\end{definition}

Now, recall the notion of the $(b,c)$-inverse.

\begin{definition}\cite{Drazin2012} Let $a,b,c\in R$. An element $y\in R$ that satisfies
	$$cay=c,~~yab=b,~~y\in bR\cap Rc,$$
	is called the $(b,c)$-inverse of $a$, and $a$ is said to be  $(b,c)$-invertible. Such an element $y$ is unique and denoted by $a^{\|(b,c)}$.
\end{definition}

If $b=c$, then $a$ is said to be invertible along $b$ \cite{Mary}, and $y$ is denoted by $a^{\|b}$.

\section{$(b,c)$-inverse of a sum with a radical element}
Let $a\in R$ with a $\{2\}$-inverse $x$, and $j_a\in J(R)$. It is well known that $(1+xj_a)^{-1}x$ is a $\{2\}$-inverse of $a+j_a$. Indeed,
\begin{eqnarray*}
	&&(1+xj_a)^{-1}x(a+j_a)(1+xj_a)^{-1}x\\
	&=&(1+xj_a)^{-1}(xa+xj_a)(1+xj_a)^{-1}x\\
	&=&(1+xj_a)^{-1}(xa+xaxj_a)(1+xj_a)^{-1}x\\
	&=&(1+xj_a)^{-1}xa(1+xj_a)(1+xj_a)^{-1}x\\
	&=&(1+xj_a)^{-1}xax=(1+xj_a)^{-1}x.
\end{eqnarray*}
A Similar argument shows that if $y$ is a $\{2\}$-inverse of $a+j_a$, then $(1-yj_a)^{-1}y$ is a $\{2\}$-inverse of $a$.
Define   mappings
\begin{eqnarray*}
	\varphi:~a\{2\}~&\longrightarrow&~(a+j_a)\{2\}\\
	x&\longmapsto& (1+xj_a)^{-1}x,
\end{eqnarray*}
and
\begin{eqnarray*}
	\psi:~(a+j_a)\{2\}~&\longrightarrow&~a\{2\}\\
	y&\longmapsto& (1-yj_a)^{-1}y.
\end{eqnarray*}
Since
\begin{eqnarray*}
	\psi\varphi(x)&=&\psi((1+xj_a)^{-1}x)=(1-(1+xj_a)^{-1}xj_a)^{-1}(1+xj_a)^{-1}x\\
	&=&(1+xj_a-(1+xj_a)(1+xj_a)^{-1}xj_a)^{-1}x=x,
\end{eqnarray*}
and $\varphi\psi(y)=y$, it follows that $\varphi$ is a bijection and  $\varphi^{-1}=\psi$.

\begin{lemma}\label{bc3.1}\emph{\cite{HuyPu}} Let $a\in R$ and $j_a\in J(R)$. If $a$ is regular with a reflexive inverse $a^+$, then $a+j_a$ is regular if and only if $$(1-aa^+)j_a(1+a^+j_a)^{-1}(1-a^+a)=0.$$ In this case, $(1+a^+j_a)^{-1}a^+$ is a reflexive inverse of $a+j_a$.
\end{lemma}

Let $a,~b,~c\in R$. Recall that if $a^{\|(b,c)}$ exists, then $b$ and $c$ are regular \cite{Wanglong2017}. In the following theorem, we denote reflexive inverses of $b$ and $c$ by $b^+$ and $c^+$, respectively.

\begin{theorem}\label{bc3.3} Let $a,~b,~c\in R$ and $j_a,~j_b,~j_c\in J(R)$. If $a^{\|(b,c)}$ exists, then the following statements are equivalent:
	\begin{itemize}
		\item [{\rm (1)}] $(a+j_a)^{\|(b+j_b,c+j_c)}$ exists;
		\item [{\rm (2)}] $b+j_b$ and $c+j_c$ are regular;
		\item [{\rm (3)}] $(1-bb^+)j_b(1+b^+j_b)^{-1}(1-b^+b)=0$ and $(1-cc^+)j_c(1+c^+j_c)^{-1}(1-c^+c)=0$.
	\end{itemize}
	In this case, $$(a+j_a)^{\|(b+j_b,c+j_c)}=(1+j_bb^+)a^{\|(b,c)}(1+ja^{\|(b,c)})^{-1}(1+c^+j_c),$$
	where $$j=j_a+aj_bb^++c^+j_ca+j_aj_bb^++c^+j_cj_a+c^+j_caj_bb^++c^+j_cj_aj_bb^+.$$
	Moreover, $a^{\|(b+j_b,c+j_c)}$ exists and
	$$a^{\|(b+j_b,c+j_c)}=(1+j_bb^+)(1+a^{\|(b,c)}(aj_bb^++c^+j_ca+c^+j_caj_bb^+))^{-1}a^{\|(b,c)}(1+c^+j_c),$$ $$(a+j_a)^{\|(b+j_b,c+j_c)}=\varphi(a^{\|(b+j_b,c+j_c)})=(1+a^{\|(b+j_b,c+j_c)}j_a)^{-1}a^{\|(b+j_b,c+j_c)}.$$
\end{theorem}
\begin{proof} (1)$\Rightarrow$(2): It is immediate from \cite{Wanglong2017}.
	
	(2)$\Rightarrow$(3): It follows from Lemma \ref{bc3.1}.
	
	(3)$\Rightarrow$(1): Since  $bb^+a^{\|(b,c)}=a^{\|(b,c)}$, we have
	\begin{eqnarray*}
		&&(1+j_bb^+)a^{\|(b,c)}=(1+j_bb^+)bb^+a^{\|(b,c)}\\
		&=&(bb^++j_bb^+)a^{\|(b,c)}=(b+j_b)b^+a^{\|(b,c)} \in (b+j_b)R,
	\end{eqnarray*}
	which implies $$(1+j_bb^+)a^{\|(b,c)}(1+ja^{\|(b,c)})^{-1}(1+c^+j_c)\in (b+j_b)R.$$
	Using the equality $a^{\|(b,c)}(1+ja^{\|(b,c)})^{-1}=(1+a^{\|(b,c)}j)^{-1}a^{\|(b,c)}$ and the fact that $a^{\|(b,c)}c^+c=a^{\|(b,c)}$, we also obtain $$(1+j_bb^+)a^{\|(b,c)}(1+ja^{\|(b,c)})^{-1}(1+c^+j_c)\in R(c+j_c).$$
	Since  $(1-bb^+)j_b(1+b^+j_b)^{-1}(1-b^+b)=0$, it follows
	\begin{eqnarray*}
		b+j_b&=&b+j_bb^+b+j_b(1-b^+b)\\
		&=&(1+j_bb^+)b+(1+j_bb^+)(1+j_bb^+)^{-1}j_b(1-b^+b)\\
		&=&(1+j_bb^+)(b+(1+j_bb^+)^{-1}j_b(1-b^+b))\\
		&=&(1+j_bb^+)(b+bb^+(1+j_bb^+)^{-1}j_b(1-b^+b))\\
		&=&(1+j_bb^+)b(1+(1+b^+j_b)^{-1}b^+j_b(1-b^+b)),
	\end{eqnarray*}
	
	and
	
	\begin{eqnarray*}
		&&a^{\|(b,c)}(1+c^+j_c)(a+j_a)(b+j_b)\\
		&=&a^{\|(b,c)}(1+c^+j_c)(a+j_a)(1+j_bb^+)b(1+(1+b^+j_b)^{-1}b^+j_b(1-b^+b))\\
		&=&(b+a^{\|(b,c)}jb)(1+(1+b^+j_b)^{-1}b^+j_b(1-b^+b)).
	\end{eqnarray*}
	Consequently,
	\begin{eqnarray*}
		&&(1+j_bb^+)a^{\|(b,c)}(1+ja^{\|(b,c)})^{-1}(1+c^+j_c)(a+j_a)(b+j_b)\\
		&=&(1+j_bb^+)(1+a^{\|(b,c)}j)^{-1}a^{\|(b,c)}(1+c^+j_c)(a+j_a)(b+j_b)\\
		&=&(1+j_bb^+)(1+a^{\|(b,c)}j)^{-1}(b+a^{\|(b,c)}jb)(1+(1+b^+j_b)^{-1}b^+j_b(1-b^+b))\\
		&=&(1+j_bb^+)b(1+(1+b^+j_b)^{-1}b^+j_b(1-b^+b))=b+j_b.
	\end{eqnarray*}
	A similar argument shows that $$(c+j_c)(a+j_a)(1+j_bb^+)a^{\|(b,c)}(1+ja^{\|(b,c)})^{-1}(1+c^+j_c)=c+j_c.$$
	Therefore, $a+j_a$ is $(b+j_b,c+j_c)$-invertible and
	$$(a+j_a)^{\|(b+j_b,c+j_c)}=(1+j_bb^+)a^{\|(b,c)}(1+ja^{\|(b,c)})^{-1}(1+c^+j_c).$$
	
	In this case, by taking $j_a=0$,  we conclude that $a^{\|(b+j_b,c+j_c)}$ exists and  is given by
	$$a^{\|(b+j_b,c+j_c)}=(1+j_bb^+)(1+a^{\|(b,c)}(aj_bb^++c^+j_ca+c^+j_caj_bb^+))^{-1}a^{\|(b,c)}(1+c^+j_c).$$
	We now claim $(a+j_a)^{\|(b+j_b,c+j_c)}=\varphi(a^{\|(b+j_b,c+j_c)}).$ Since $a^{\|(b+j_b,c+j_c)}\in (b+j_b)R\cap R(c+j_c)$, it follows
	\begin{eqnarray*}
		&&\varphi(a^{\|(b+j_b,c+j_c)})=(1+a^{\|(b+j_b,c+j_c)}j_a)^{-1}a^{\|(b+j_b,c+j_c)}\\
		&=&a^{\|(b+j_b,c+j_c)}(1+j_a a^{\|(b+j_b,c+j_c)})^{-1}\in (b+j_b)R\cap R(c+j_c).
	\end{eqnarray*}
	A direct computation then yields
	\begin{eqnarray*}
		&&(c+j_c)(a+j_a)\varphi(a^{\|(b+j_b,c+j_c)})\\
		&=&(c+j_c)(a+j_a)a^{\|(b+j_b,c+j_c)}(1+j_a a^{\|(b+j_b,c+j_c)})^{-1}\\
		&=&(c+j_c)(aa^{\|(b+j_b,c+j_c)}+j_a a^{\|(b+j_b,c+j_c)})(1+j_a a^{\|(b+j_b,c+j_c)})^{-1}\\
		&=&(c+j_c)(1+j_a a^{\|(b+j_b,c+j_c)})(1+j_a a^{\|(b+j_b,c+j_c)})^{-1}=c+j_c
	\end{eqnarray*}
	and $\varphi(a^{\|(b+j_b,c+j_c)})(a+j_a)(b+j_b)=(b+j_b)$, which completes the proof.
\end{proof}

\begin{corollary}\label{bc3.2}  Let $a,~b,~c\in R$ and $j_a\in J(R)$. Then $a^{\|(b,c)}$ exists if and only if $(a+j_a)^{\|(b,c)}$ exists. In this case, $(a+j_a)^{\|(b,c)}=\varphi(a^{\|(b,c)})=(1+a^{\|(b,c)}j_a)^{-1}a^{\|(b,c)}$.
\end{corollary}
Given invertible elements $a,b\in R$, the classical absorption law is known as $a^{-1}(a+b)b^{-1}=a^{-1}+b^{-1}$. Such a property is extended to cases of generalized inverses.
According to \cite[Remark 3.4]{Zhouyukun2025}, we obtain the following result on the absorption law for $\{2\}$-inverses.

\begin{corollary}\label{bc3.2} Let $a\in R$ and $j_a\in J(R)$. If $a$ and $a+j_a$ possess $\{2\}$-inverses $x$ and $y$, respectively, then the following statements are equivalent:
	\begin{itemize}
		\item[\rm(1)] $x(a+(a+j_a))y=x+y=y(a+(a+j_a))x$;
		\item[\rm(2)] $x(a+(a+j_a))y=x+y$;
		\item[\rm(3)] $xj_ay=x-y$;
		\item[\rm(4)] $xR=yR$ and $Rx=Ry$;
		\item[\rm(5)] $y=\varphi(x)=(1+xj_a)^{-1}x$;
		\item[\rm(6)] There exist $b,c\in R$ such that $x=a^{\|(b,c)}$ and $y=(a+j_a)^{\|(b,c)}$.
	\end{itemize}
\end{corollary}
\begin{proof} The equivalence of conditions (1), (4), (5), and (6) follows directly  from \cite[Remark 3.4]{Zhouyukun2025}. Furthermore, Condition (2) is equivalent to $(1+xj_a)y=x$, which can be rewritten as $y=(1+xj_a)^{-1}x$, i.e., condition (5). The equivalence between (3) and (5) is clear.
\end{proof}

The Moore-Penrose inverse, group inverse and core inverse of a sum with radical element were investigated in \cite{HuyPu,LCWX}. The following corollary recovers these results and provides new conditions and  expressions.

\begin{corollary} Let $R$ be equipped with an involution $*$, $a\in R$ and $j_a\in J(R)$. If $a$ is, respectively, Moore-Penrose invertible, group invertible and/or core invertible with a reflexive inverse $a^+$, then the following statements are equivalent:
	\begin{itemize}
		\item [{\rm (1)}] $a+j_a$ is, respectively, Moore-Penrose invertible, group invertible and/or core invertible;
		\item [{\rm (2)}] $a+j_a$ is regular;
		\item [{\rm (3)}] $(1-aa^+)j_a(1+a^+j_a)^{-1}(1-a^+a)=0$.
	\end{itemize}
	In these cases, the corresponding generalized inverses are given by:
	\begin{eqnarray*}
		(a+j_a)^{\dagger}&=&(1+j_a^*(a^+)^*)a^{\dagger}(1+j_1a^{\dagger})^{-1}(1+(a^+)^*j_a^*),\\
		(a+j_a)^{\#}&=&(1+j_aa^+)a^{\#}(1+j_2a^{\#})^{-1}(1+a^+j_a),\\
		(a+j_a)^{\scriptsize\textcircled{\tiny \#}}&=&(1+j_aa^+)a^{\scriptsize\textcircled{\tiny \#}}(1+j_3a^{\scriptsize\textcircled{\tiny \#}})^{-1}(1+(a^+)^*j_a^*),
	\end{eqnarray*}
	where
	\begin{eqnarray*}
		j_1&=&j_a+aj_a^*(a^+)^*+(a^+)^*j_a^*a+j_aj_a^*(a^+)^*+(a^+)^*j_a^*j_a\\
		&&+(a^+)^*j_a^*aj_a^*(a^+)^*+(a^+)^*j_a^*j_aj_a^*(a^+)^*,\\
		j_2&=&j_a+aj_aa^++a^+j_aa+j_a^2a^++a^+j_a^2+a^+j_aaj_aa^++a^+j_a^3a^+,\\
		j_3&=&j_a+aj_aa^++(a^+)^*j_a^*a+j_a^2a^++(a^+)^*j_a^*j_a+(a^+)^*j_a^*aj_aa^++(a^+)^*j_a^*j_a^2a^+.
	\end{eqnarray*}
	
\end{corollary}
\begin{proof} It follows from \cite{Mary} that $a^{\dagger}=a^{\|a^*}$ and $a^{\#}=a^{\|a}$, and from \cite{RDD} that $a^{\scriptsize\textcircled{\tiny \#}}=a^{\|(a,a^*)}$. The rest of the proof then follows directly from Theorem \ref{bc3.3}.
\end{proof}

\begin{corollary}\label{bc3.6} Let $a\in R$ and $j_a\in J(R)$. If $a$ is Drazin invertible with ${\rm i}(a)=k$, then the following statements are equivalent:
	\begin{itemize}
		\item [{\rm (1)}] $a+j_a$ is Drazin invertible;
		\item [{\rm (2)}] $(a+j_a)^l$ is regular for some $l\geq k$;
		\item [{\rm (3)}] $(1-aa^D)j_1(1+(a^D)^lj_1)^{-1}(1-a^Da)=0$ for some $l\geq k$, where $j_1=(a+j_a)^l-a^l$.
	\end{itemize}
	In this case, $$(a+j_a)^D=(1+j_1(a^D)^l)a^D(1+ja^D)^{-1}(1+(a^D)^lj_1),$$
	where $j=j_a+aj_1(a^D)^l+(a^D)^lj_1a+j_aj_1(a^D)^l+(a^D)^lj_1j_a+(a^D)^lj_1aj_1(a^D)^l+(a^D)^lj_1j_aj_1(a^D)^l$.

\end{corollary}
\begin{proof}(1)$\Rightarrow$(2): Assume ${\rm i}(a+j_a)=t$, and take $l=max\{t,k\}$. Then $(a+j_a)^l$ is regular and $l\geq k$.
	
	(2)$\Rightarrow$(3)$\Rightarrow$(1): These follow from Theorem \ref{bc3.3} and the fact that $a^{\|a^l}$ exists for any   $l\geq k$.
\end{proof}

\begin{remark} In the  power series ring $R \llbracket{x}\rrbracket$, researchers investigated  Moore-Penrose inverses of elements (cf.\cite{Huy,Wynn1971}). The above results can be applied to characterize several generalized inverses of elements in $R \llbracket{x}\rrbracket$. We present the result for  the $(b,c)$-inverse.
	
	Let ${\sum\limits_{i\geqslant 0} a_ix^i},{\sum\limits_{i\geqslant 0} b_ix^i},{\sum\limits_{i\geqslant 0} c_ix^i}\in R \llbracket{x}\rrbracket$. The following statements are equivalent:
	\begin{itemize}
		\item [{\rm (1)}] ${\sum\limits_{i\geqslant 0} a_ix^i}$ is $({\sum\limits_{i\geqslant 0} b_ix^i},{\sum\limits_{i\geqslant 0} c_ix^i})$-invertible;
		\item [{\rm (2)}] $a_0^{\|(b_0,c_0)}$ exists with
		$$(1-b_0b_0^+)({\sum\limits_{i\geqslant 1} b_ix^i})(1+b_0^+({\sum\limits_{i\geqslant 1} b_ix^i}))^{-1}(1-b_0^+b_0)=0,$$ $$(1-c_0c_0^+)({\sum\limits_{i\geqslant 1} c_ix^i})(1+c_0^+({\sum\limits_{i\geqslant 1} c_ix^i}))^{-1}(1-c_0^+c_0)=0,$$
		where $b_0^+$ and $c_0^+$ are reflexive inverses of $b_0$ and $c_0$, respectively.
	\end{itemize}
	
\end{remark}

We now turn our attention to the idempotence of the $(b,c)$-inverse.

\begin{lemma} \emph{\cite{Zhuhaiyang2022}} \label{bc3.4} Let $a,b,c\in R$. The following statements are equivalent:
	\begin{itemize}
		\item [{\rm (1)}] $a^{\|(b,c)}$ exists, and is idempotent;
		\item [{\rm (2)}] $cab=cb$ and $1^{\|(b,c)}$ exists.
	\end{itemize}
	In this case, $a^{\|(b,c)}=1^{\|(b,c)}$.
\end{lemma}

In what follows, we explore the idempotence of $a^{\|(b,c)}$ and $(a+j_a)^{\|(b,c)}$.

\begin{corollary}\label{bc3.9}  Let $a,~b,~c\in R$ and $j_a\in J(R)$. If $a^{\|(b,c)}$ exists, then the following statements are equivalent:
	\begin{itemize}
		\item [{\rm (1)}] $(a+j_a)^{\|(b,c)}$ is idempotent;
		\item [{\rm (2)}] $(a^{\|(b,c)})^2=a^{\|(b,c)}(a+j_a)a^{\|(b,c)}$;
		\item [{\rm (3)}] $a+j_a=c^+cb b^++(1-c^+c)(a+j_a)+c^+c(a+j_a)(1-b b^+)$.
	\end{itemize}
\end{corollary}

\begin{proof} Using Corollary \ref{bc3.2}, we have that $(a+j_a)^{\|(b,c)}$  exists. Furthermore, it follows from Lemma \ref{bc3.4} that $(a+j_a)^{\|(b,c)}$ is idempotent if and only if $c(a+j_a)b=cb$.
	
	(1)$\Leftrightarrow$(2): According to the definition of the $(b,c)$-inverse, there exist $s,t\in R$ such that $a^{\|(b,c)}=sc=bt$. Moreover, $a^{\|(b,c)}ab=b$ and $caa^{\|(b,c)}=c$. Consequently, the equality $c(a+j_a)b=cb$ holds if and only if $(a^{\|(b,c)})^2=a^{\|(b,c)}(a+j_a)a^{\|(b,c)}$.
	
	(1)$\Leftrightarrow$(3): A straightforward computation yields the identity
	$$a+j_a=c^+c(a+j_a)b b^++(1-c^+c)(a+j_a)+c^+c(a+j_a)(1-b b^+).$$ In this decomposition, the equality $c^+c(a+j_a)b b^+=c^+cb b^+$ is equivalent to the equality $c(a+j_a)b=cb$. Hence condition (1) holds if and only if condition (3) holds.
\end{proof}

Given  $b,c\in R$, the product $cb$ is called a trace product if $cbR=cR$ and $Rcb=Rb$. According to a result of Clifford \cite[Proposition 2.3.7]{Howie1995}, this is equivalent to the existence of an idempotent $e$ such that $bR=eR$ and $Rc=Re$. Obviously, $cb$ is  a trace product if and only if $1^{\|(b,c)}$ exists, in which case $1^{\|(b,c)}=e$.

\begin{corollary}\label{bc3.10}  Let $a,~b,~c\in R$ and $j_a\in J(R)$. Then the following statements are equivalent:
	\begin{itemize}
		\item [{\rm (1)}] $a^{\|(b,c)}$ exists, and the elements $a^{\|(b,c)}$,  $(a+j_a)^{\|(b,c)}$  are idempotent;
		\item [{\rm (2)}] $cb$ is a trace product, $a=c^+cb b^++(1-c^+c)a+c^+ca(1-b b^+)$ and $j_a=(1-c^+c)j_a+c^+cj_a(1-b b^+)$.
	\end{itemize}
	In this case, $a^{\|(b,c)}=(a+j_a)^{\|(b,c)}=1^{\|(b,c)}$.
\end{corollary}
\begin{proof}(1)$\Rightarrow$(2): Since $a^{\|(b,c)}$ is idempotent, it follows from Lemma \ref{bc3.4} that $cab=cb$ and  $1^{\|(b,c)}$ exists. Thus,  $cb$ is a trace product. Using Corollary \ref{bc3.9}, we have $$a=c^+cb b^++(1-c^+c)a+c^+ca(1-b b^+)$$
	and $$a+j_a=c^+cb b^++(1-c^+c)(a+j_a)+c^+c(a+j_a)(1-b b^+),$$
	which imply  $j_a=(1-c^+c)j_a+c^+cj_a(1-b b^+)$.
	
	(2)$\Rightarrow$(1): The equality $a=c^+cb b^++(1-c^+c)a+c^+ca(1-b b^+)$ implies $cab=cb$. Together with the condition that $cb$ is a trace product,  this yields  that $a^{\|(b,c)}$ exists and is idempotent. Then, by Corollary \ref{bc3.9},  $(a+j_a)^{\|(b,c)}$ is also idempotent.
	
\end{proof}

Recall that an element $a\in R$ is said to be clean if there exists an idempotent $e$ and a unit $u$ such that $a=e+u$. Clean elements are closely related to generalized inverses (cf. \cite{Drazin2012,Mary2020,ZhuPa2018,ZZPa2019}). Furthermore, if $aR\cap eR =\{0\}$, then $a$ is said to be special clean. It was proved in \cite{Mary2020} that $a$ is special clean if and only if $a$ has a reflexive inverse $z$ which is group invertible. The next corollary provides characterizations of special clean property of a sum with radical element.

Suppose that $a$ and $a+j_a$ are regular for some $j_a\in J(R)$. Then the map $\varphi$ induces a bijection between $a\{1,2\}$ and $(a+j_a)\{1,2\}$. Indeed, for arbitrary $x\in a\{1,2\}$, Lemma \ref{bc3.1} implies that $\varphi(x)\in (a+j_a)\{1,2\}$. Similarly,  for arbitrary $y\in (a+j_a)\{1,2\}$, we have $\psi(y)\in a\{1,2\}$. Since $\varphi$ and $\psi$ are mutually inverse mappings between $a\{2\}$ and $(a+j_a)\{2\}$, it follows that  $\varphi$ restricts to a bijection between $a\{1,2\}$ and $(a+j_a)\{1,2\}$. The notation $R^{\#}$ denotes the set of all group invertible elements in $R$.

\begin{corollary} Let $a\in R$ and $j_a\in J(R)$. If $a$ is special clean with a reflexive inverse $a^+$, then the following statements are equivalent:
	\begin{itemize}
		\item [{\rm (1)}] $a+j_a$ is  special clean;
		\item [{\rm (2)}] $a+j_a$ is regular;
		\item [{\rm (3)}] $(1-aa^+)j_a(1+a^+j_a)^{-1}(1-a^+a)=0$.
	\end{itemize}
	In this case, $(1+zj_a)^{-1}z$ is  group invertible for arbitrary $z\in a\{1,2\}\cap R^{\#}$. The map  $\varphi$ induces a bijection between $a\{1,2\}\cap R^{\#}$ and $(a+j_a)\{1,2\}\cap R^{\#}$.
\end{corollary}
\begin{proof} (1)$\Rightarrow$(2) is obvious, and (2)$\Leftrightarrow$(3) directly follows from Lemma \ref{bc3.1}.
	
	(2)$\Rightarrow$(1): Suppose that $a$ has a reflexive inverse $z$ which is group invertible. According to Lemma \ref{bc3.1}, we have $(1+zj_a)^{-1}z\in (a+j_a)\{1,2\}$. Since  $z$  is group invertible, it follows that $(1+zj_a)^{-1}z$ is also group invertible, and its group inverse is given by  $((1+zj_a)^{-1}z)^{\#}=z^{\#}+zz^{\#}j_azz^{\#}$. Consequently, $a+j_a$ admits a reflexive inverse that is group invertible; that is, $a+j_a$ is special clean by \cite{Mary2020}.
	
	In this case,  $z\in a\{1,2\}$  is group invertible if and only if $\varphi(z)$   is group invertible. Hence, $\varphi$ induces a bijection between $a\{1,2\}\cap R^{\#}$ and $(a+j_a)\{1,2\}\cap R^{\#}$.
\end{proof}

Recall that if $a\in R$ has a clean decomposition $a=e+u$ with $ea=ae$, then $a$ is said to be strongly clean, and this decomposition is called a strongly clean decomposition.

\begin{corollary}\label{bc3.8} Let $a\in R$ and $j_a\in J(R)$. Then the following statements are equivalent:
	\begin{itemize}
		\item [{\rm (1)}] $a+j_a$ is  strongly clean;
		\item [{\rm (2)}] $a$ has the clean decomposition $a=\bar{e}+u$ such that $a^{\|e}$ exists with $aa^{\|e}(a+j_a)(1-aa^{\|e})=(1-a^{\|e}a)(a+j_a)a^{\|e}a=0$, where $e+\bar{e}=1$;
		\item [{\rm (3)}] $a$ has the clean decomposition $a=\bar{e}+u$ such that $1^{\|(ueu^{-1},e)}$ exists with $1^{\|(ueu^{-1},e)}(a+j_a)(1-1^{\|(ueu^{-1},e)})=(1-1^{\|(ueu^{-1},e)})u(a+j_a)u^{-1}1^{\|(ueu^{-1},e)}=0$.
	\end{itemize}
	
\end{corollary}

\begin{proof}  (1)$\Rightarrow$(2): There exists a strongly clean decomposition $a+j_a=\bar{e}+v$, where $\bar{e}$ is idempotent and $v$ is a unit. Take $u=v-j_a$, thus $a$ has the clean decomposition $a=\bar{e}+u$. It is easy to verify that $(a+j_a)^{\|e}=ev^{-1}$ and $(a+j_a)(a+j_a)^{\|e}=(a+j_a)^{\|e}(a+j_a)=e$, which implies that $a^{\|e}$ exists and $e=(1+j_aa^{\|e})aa^{\|e}(1+j_aa^{\|e})^{-1}=(1+a^{\|e}j_a)^{-1}a^{\|e}a(1+a^{\|e}j_a)$ by Corollary \ref{bc3.2}. Thus, $\bar{e}=(1-aa^{\|e})(1+j_aa^{\|e})^{-1}=(1+a^{\|e}j_a)^{-1}(1-a^{\|e}a)$. Since $1-aa^{\|e}=\bar{e}(1+j_aa^{\|e})$ and $aa^{\|e}e=aa^{\|e}$, we have $aa^{\|e}(a+j_a)(1-aa^{\|e})=aa^{\|e}e(a+j_a)\bar{e}(1+j_aa^{\|e})=0$. Similarly, $(1-a^{\|e}a)(a+j_a)a^{\|e}a=0$.
	
	(2)$\Rightarrow$(1): We claim that $(a+j_a)^{\|e}=(1+a^{\|e}j_a)^{-1}a^{\|e}$ commutes with $a+j_a$. It suffices to prove $a^{\|e}(a+j_a)(1+j_a a^{\|e})=(1+a^{\|e}j_a)(a+j_a)a^{\|e}$, which can be simplified to
	\begin{equation}\label{bc3.5}
		a^{\|e}(a+j_a)+a^{\|e}aj_aa^{\|e}=(a+j_a)a^{\|e}+a^{\|e}j_aaa^{\|e}.
	\end{equation}
	Since $aa^{\|e}(a+j_a)(1-aa^{\|e})=0$,  the left side of (\ref{bc3.5})  equals $a^{\|e}(a+j_a)aa^{\|e}+a^{\|e}aj_aa^{\|e}$. Similarly, the right side of (\ref{bc3.5}) equals $a^{\|e}a(a+j_a)a^{\|e}+a^{\|e}j_aaa^{\|e}$. Thus, the equality (\ref{bc3.5}) holds. This together with $(a+j_a)^{\|e}(a+j_a)R=eR$ and $R(a+j_a)(a+j_a)^{\|e}=Re$  implies $(a+j_a)^{\|e}(a+j_a)=e$. Hence, $a+j_a=\bar{e}+(u+j_a)$ is a strongly clean decomposition.
	
	(2)$\Leftrightarrow$(3): Noting $eae=eue$, we conclude that $a^{\|e}$ exists if and only if $1^{\|(ueu^{-1},e)}$. Furthermore, $1^{\|(ueu^{-1},e)}=aa^{\|e}$ and $u^{-1}1^{\|(ueu^{-1},e)}u=a^{\|e}a$.
\end{proof}

The following example shows that $a$ need not be strongly clean in the   above corollary.

\begin{example} Let $R={\mathbb T}_2(\mathbb{Z})$ be the upper triangular ring. Take $a=\left(\begin{matrix}
		2&2\\
		0&-1
	\end{matrix}
	\right)$ and $j_a=\left(\begin{matrix}
		0&1\\
		0&0
	\end{matrix}
	\right)$. There exists a clean decomposition $a=\left(\begin{matrix}
		1&1\\
		0&0
	\end{matrix}
	\right)+\left(\begin{matrix}
		1&1\\
		0&-1
	\end{matrix}
	\right)$ satisfying Corollary \ref{bc3.8}(2). Thus, $a+j_a$ is strongly clean.
	
	However, it is easy to verify that $a$ is not strongly clean. Suppose, to the contrary, $a=e+u$ is a strongly clean decomposition, where $e=\left(\begin{matrix}
		x_1&x_2\\
		0&x_3
	\end{matrix}
	\right)$ is an idempotent and $u=\left(\begin{matrix}
		2-x_1&2-x_2\\
		0&-1-x_3
	\end{matrix}
	\right)$  is a unit. For $e$ to be idempotent and $u$ to be invertible, the integers  $x_1, x_3$ must be either $0$ or $1$, and the product $(2-x_1)(-1-x_3)$ must equal $1$ or $-1$. These conditions force  $x_1=1$ and $x_3=0$. Moreover, the commutativity condition $eu=ue$  implies that $x_2$ must equal $\dfrac{2}{3}$, which is not an integer. It is a contradiction.
\end{example}

\section{Applications to dual matrices}

Let $\widehat{A}\in \mathbb{D}^{m\times n}$, $\widehat{B}\in \mathbb{D}^{n\times s}$, $\widehat{C}\in \mathbb{D}^{t\times m}$. The dual matrix $\widehat{X}\in \mathbb{D}^{n\times m}$ that satisfies
$$\widehat{C}\widehat{A}\widehat{X}=\widehat{C},~~\widehat{X}\widehat{A}\widehat{B}=\widehat{B},~~\widehat{X}=\widehat{B}\widehat{Y_1}=\widehat{Y}_2\widehat{C} ~\text{for~some~}\widehat{Y}_1\in \mathbb{D}^{s\times m},~\widehat{Y}_2\in \mathbb{D}^{n\times t},$$
is called the $(\widehat{B},\widehat{C})$-inverse of $\widehat{A}$, and $\widehat{A}$ is said to be  $(\widehat{B},\widehat{C})$-invertible. Such a dual matrix $\widehat{X}$ is unique (the proof is similar to the case of ring elements) and denoted by $\widehat{A}^{\|(\widehat{B},\widehat{C})}$. If $\widehat{B}=\widehat{C}$, then $\widehat{A}$ is said to be invertible along $\widehat{B}$, and $\widehat{X}$ is denoted by $\widehat{A}^{\|\widehat{B}}$.

Let $\widehat{A}=A+\varepsilon A_0\in \mathbb{D}^{m\times n}$ and let $A^+$ be a reflexive inverse of $A$. By Lemma \ref{bc3.1},  $\widehat{A}$ is regular if and only if $(I-AA^+)A_0(I-A^+A)=0$. According to \cite{BK1979}, this is equivalent to the existence of matrices $A_1\in \mathbb{C}^{m\times m}$ and $A_2\in \mathbb{C}^{n\times n}$  such that $A_0=A_1A+AA_2$, which  in turn is equivalent to $\widehat{A}=(I+\varepsilon A_1)A(I+\varepsilon A_2)$. We now apply Theorem \ref{bc3.3} to establish the main result of this section.

\begin{proposition}\label{bc2.2}  Let $\widehat{A}=A+\varepsilon A_0\in \mathbb{D}^{m\times n}$, $\widehat{B}=B+\varepsilon B_0\in \mathbb{D}^{n\times s}$, $\widehat{C}=C+\varepsilon C_0\in \mathbb{D}^{t\times m}$. Suppose that $B$ and $C$ have reflexive inverses $B^+$ and $C^+$, respectively. Then the following statements are equivalent:
	\begin{itemize}
		\item [{\rm (1)}] $\widehat{A}^{\|(\widehat{B},\widehat{C})}$ exists;
		\item [{\rm (2)}] $A^{\|(\widehat{B},\widehat{C})}$ exists;
		\item [{\rm (3)}] $A^{\|(B,C)}$ exists with $(I-BB^+)B_0(I-B^+B)=0$, $(I-CC^+)C_0(I-C^+C)=0$.
	\end{itemize}
	In this case,
	\begin{eqnarray*}
		A^{\|(\widehat{B},\widehat{C})}&=&A^{\|(B,C)}\\
		&&+\varepsilon(I-A^{\|(B,C)}A)B_1B^{\|(B,C)}+\varepsilon A^{\|(B,C)}C_2(I-AA^{\|(B,C)}),\\
		\widehat{A}^{\|(\widehat{B},\widehat{C})}&=&A^{\|(B,C)}-\varepsilon A^{\|(B,C)}A_0A^{\|(B,C)}\\
		&&+\varepsilon(I-A^{\|(B,C)}A)B_1B^{\|(B,C)}+\varepsilon A^{\|(B,C)}C_2(I-AA^{\|(B,C)}),
	\end{eqnarray*}
	where $B_0=B_1B+BB_2$, $C_0=C_1C+CC_2$.
\end{proposition}
\begin{proof} If the dual matrices are not square, they can be made square   by adding  adequately zero columns or zero rows. Hence, the equivalence among statements (1), (2) and (3) follows directly from Theorem \ref{bc3.3}.  We now derive explicit expressions for $A^{\|(\widehat{B},\widehat{C})}$ and $\widehat{A}^{\|(\widehat{B},\widehat{C})}$. According to the expressions presented in Theorem \ref{bc3.3}, we obtain
	\begin{eqnarray*}
		&&A^{\|(\widehat{B},\widehat{C})}\\
		&=&(I+\varepsilon B_0B^+)(I+A^{\|(B,C)}(\varepsilon AB_0B^++\varepsilon C^+C_0A))^{-1}A^{\|(B,C)}(I+\varepsilon C^+C_0)\\
		&=&(I+\varepsilon B_0B^+)(I-A^{\|(B,C)}(\varepsilon AB_0B^++\varepsilon C^+C_0A))A^{\|(B,C)}(I+\varepsilon C^+C_0)\\
		&=&A^{\|(B,C)}+\varepsilon(I-A^{\|(B,C)}A)B_0B^+A^{\|(B,C)}+\varepsilon A^{\|(B,C)}C^+C_0(I-AA^{\|(B,C)}).
	\end{eqnarray*}
	Substituting $B_0=B_1B+BB_2$, $C_0=C_1C+CC_2$ in the last expression, we further simplify:
	\begin{eqnarray*}
		A^{\|(\widehat{B},\widehat{C})}&=&A^{\|(B,C)}\\
		&&+\varepsilon(I-A^{\|(B,C)}A)B_1B^{\|(B,C)}+\varepsilon A^{\|(B,C)}C_2(I-AA^{\|(B,C)}),\\
		\widehat{A}^{\|(\widehat{B},\widehat{C})}&=&\varphi(A^{\|(\widehat{B},\widehat{C})})=(I-\varepsilon A^{\|(\widehat{B},\widehat{C})}A_0)A^{\|(\widehat{B},\widehat{C})}\\
		&=&A^{\|(B,C)}-\varepsilon A^{\|(B,C)}A_0A^{\|(B,C)}\\
		&&+\varepsilon(I-A^{\|(B,C)}A)B_1B^{\|(B,C)}+\varepsilon A^{\|(B,C)}C_2(I-AA^{\|(B,C)}).
	\end{eqnarray*}
\end{proof}

\begin{remark} In Proposition \ref{bc2.2}(3), for any $B^-\in B\{1\}$, the condition $(I-BB^+)B_0(I-B^+B)=0$ is equivalent to $(I-BB^-)B_0(I-B^-B)=0$. Thus, the reflexive inverses  $B^+$ and $C^+$ can be replaced by arbitrary  $B^-\in B\{1\}$ and $C^-\in C\{1\}$.
	
	Noting that $\mathbb{Q}^{n\times n}$ is isomorphic to the ring $\{\left(\begin{matrix}
		A&A_0\\
		0&A
	\end{matrix}
	\right):~A,A_0\in \mathbb{R}^{n\times n}\}$, one may apply  the method  described in \cite{Mary2012,2016kewangchen} to provide an alternative proof of Proposition \ref{bc2.2}.
\end{remark}

Let $A\in \mathbb{R}^{m\times n}$. For subspaces $E\subset \mathbb{R}^m$ and $F\subset  \mathbb{R}^n$, if there exists $X\in \mathbb{R}^{n\times m}$ such that $$XAX=X,~~\mathcal{R}(X)=E,~~\mathcal{N}(X)=F,$$ then $X$ is unique and  is denoted by $A^{(2)}_{E,F}$ \cite{BG}. Let $B,C$ be real matrices. The existence of $A^{(2)}_{\mathcal{R}(B),\mathcal{N}(C)}$ is equivalent to that of $A^{\|(B,C)}$, since the condition $\mathcal{N}(X)=\mathcal{N}(C)$ is equivalent to $\mathcal{R}(X^{\mathrm T})=\mathcal{R}(C^{\mathrm T})$. In this case, $A^{(2)}_{\mathcal{R}(B),\mathcal{N}(C)}$ coincides with  $A^{\|(B,C)}$. This property extends naturally to dual matrices. For $\widehat{A}\in \mathbb{R}^{m\times n}$, the range and the null space \cite{Zhongjin2023} of $\widehat{A}$ are defined as follows:
\begin{eqnarray*}
	&\mathcal{R}(\widehat{A})&=\{\widehat{A}\hat{\alpha}:~\hat{\alpha}\in \mathbb{D}^{n}\};\\
	&\mathcal{N}(\widehat{A})&=\{\hat{\beta}\in \mathbb{D}^{n}:~\widehat{A}\hat{\beta}=0\}.		
\end{eqnarray*}

\begin{lemma}\label{bc4.5} Let  $\widehat{A}=A+\varepsilon A_0\in \mathbb{D}^{m\times n}$ and $\widehat{B}=B+\varepsilon B_0\in \mathbb{D}^{m\times n}$. If $\widehat{B}$ is regular and $\mathcal{N}(\widehat{A})=\mathcal{N}(\widehat{B})$, then $\widehat{A}$ is regular and $\mathcal{R}(\widehat{A}^{\mathrm T})=\mathcal{R}(\widehat{B}^{\mathrm T})$.
\end{lemma}
\begin{proof} Since $\widehat{B}$ is regular, there exists an idempotent dual matrix $\widehat{E}=E+\varepsilon E_0$ such that $\mathcal{N}(\widehat{B})=\mathcal{N}(\widehat{E})=\mathcal{R}(I-\widehat{E})$. It follows that $\widehat{A}(I-\widehat{E})=0$. Multiplying by $I-AA^-$ on the left (where $A^-\in A\{1\}$), we obtain $(I-AA^-)A_0(I-E)=0$. For any $\alpha\in \mathbb{R}^n$, we have $\widehat{A}(\varepsilon(I-A^-A)\alpha)=0$ and $\widehat{E}(\varepsilon(I-A^-A)\alpha)=0$, which implies $(I-E)(I-A^-A)=I-A^-A$. This together with $(I-AA^-)A_0(I-E)=0$ implies $(I-AA^-)A_0(I-A^-A)=(I-AA^-)A_0(I-E)(I-A^-A)=0$. Therefore, $\widehat{A}$ is regular and $\mathcal{R}(\widehat{A}^{\mathrm T})=\mathcal{R}(\widehat{B}^{\mathrm T})$.
\end{proof}

Let $\widehat{A}\in \mathbb{D}^{m\times n}$, $\widehat{B}\in \mathbb{D}^{n\times s}$, $\widehat{C}\in \mathbb{D}^{t\times m}$. If there exists $\widehat{X}\in \mathbb{D}^{n\times m}$ such that
$$\widehat{X}\widehat{A}\widehat{X}=\widehat{X},~~\mathcal{R}(\widehat{X})=\mathcal{R}(\widehat{B}),~~\mathcal{N}(\widehat{X})=\mathcal{N}(\widehat{C}),$$ then $\widehat{X}$ is unique. In fact, suppose that $\widehat{X}_1$ and $\widehat{X}_2$ both satisfy above equations. By Lemma \ref{bc4.5}, it follows $\mathcal{R}(\widehat{X}_1)=\mathcal{R}(\widehat{X}_2)$ and $\mathcal{R}(\widehat{X}_1^{\mathrm T})=\mathcal{R}(\widehat{X}_2^{\mathrm T})$. The relation $\mathcal{R}(\widehat{X}_1)=\mathcal{R}(\widehat{X}_2)$ implies $\widehat{X}_1\widehat{A}\widehat{X}_2=\widehat{X}_2$, while $\mathcal{R}(\widehat{X}_1^{\mathrm T})=\mathcal{R}(\widehat{X}_2^{\mathrm T})$ implies $\widehat{X}_1\widehat{A}\widehat{X}_2=\widehat{X}_1$. Thus, $\widehat{X}_1=\widehat{X}_2$. We denote $\widehat{X}$ by the symbol $\widehat{A}^{(2)}_{\mathcal{R}(\widehat{B}),\mathcal{N}(\widehat{C})}$.

Using Lemma \ref{bc4.5}, we immediately deduce the equivalence between the existence of  $A^{(2)}_{\mathcal{R}(\widehat{B}),\mathcal{N}(\widehat{C})}$ and that of $\widehat{A}^{\|(\widehat{B},\widehat{C})}$.

\begin{proposition}  Let $\widehat{A}\in \mathbb{D}^{m\times n}$, $\widehat{B}\in \mathbb{D}^{n\times s}$, $\widehat{C}\in \mathbb{D}^{t\times m}$. Then the following statements are equivalent:
	\begin{itemize}
		\item [{\rm (1)}] $\widehat{A}^{\|(\widehat{B},\widehat{C})}$ exists;
		\item [{\rm (2)}] $\widehat{A}^{(2)}_{\mathcal{R}(\widehat{B}),\mathcal{N}(\widehat{C})}$ exists.
	\end{itemize}
	In this case, $\widehat{A}^{(2)}_{\mathcal{R}(\widehat{B}),\mathcal{N}(\widehat{C})}=\widehat{A}^{\|(\widehat{B},\widehat{C})}$.
\end{proposition}

Let $A\in \mathbb{R}^{n\times n}$ and $A=GH$ be the full rank decomposition. A well-known result states that $A$ is group invertible if and only if $HG$ is invertible, in which case $A^{\#}=G(HG)^{-2}H$. Besides, the Moore-Penrose inverse can be expressed as $A^{\dagger}=H^{\mathrm T}(H H^{\mathrm T})^{-1}(G^{\mathrm T}G)^{-1}G^{\mathrm T}.$ Now, we extend these results  to dual matrices.

\begin{proposition}\label{bc4.6} Let $\widehat{A}=A+\varepsilon A_0\in \mathbb{D}^{m\times n}$, $\widehat{D}=D+\varepsilon D_0\in \mathbb{D}^{n\times m}$, and $D=GH$ be the full rank decomposition. The following statements are equivalent:
	\begin{itemize}
		\item [{\rm (1)}] $\widehat{A}^{\|\widehat{D}}$ exists;
		\item [{\rm (2)}] $A^{\|D}$ exists with $(1-DD^-)D_0(1-D^-D)=0$ for some $D^-\in D\{1\}$;
		\item [{\rm (3)}] $\widehat{D}=(I+\varepsilon D_1)D(I+\varepsilon D_2)$ for some real matrices $D_1,D_2$, and $\widehat{Z}= H(I+\varepsilon D_2)\widehat{A}(I+\varepsilon D_1)G$  is invertible.
	\end{itemize}
	In this case,
	$\widehat{A}^{\|\widehat{D}}=(I+\varepsilon D_1)G\widehat{Z}^{-1}H(I+\varepsilon D_2).$
\end{proposition}
\begin{proof} The equivalence between (1) and (2)  follows from Proposition \ref{bc2.2}.
	
	The equivalence between (1) and (3)  follows from \cite[Theorem 4.2]{Zhouyukun2024}. We now provide  a direct proof. Suppose that the rank of $D$ is $r$. There exist $S\in \mathbb{R}^{r\times n},~T\in \mathbb{R}^{m\times r}$ such that $SG=HT=I$.
	
	(1)$\Rightarrow$(3): It is clear from the discussion at the beginning of Section 4 that $\widehat{D}$ can be expressed as $\widehat{D} = (I + \varepsilon D_1) D (I + \varepsilon D_2)$ for some real matrices $D_1$ and $D_2$.  Moreover, there exists $\widehat{L} \in \mathbb{D}^{m \times m}$ such that $\widehat{A}^{\|\widehat{D}} = \widehat{D} \widehat{L}$. We claim that $\widehat{Z}$ is invertible and that its inverse is given by $\widehat{Z}^{-1} = S(I - \varepsilon D_1) \widehat{A}^{\|\widehat{D}} (I - \varepsilon D_2) T$. Noting $S(I-\varepsilon D_1)(I+\varepsilon D_1)G=I$, we have
	\begin{eqnarray*}
		&&\widehat{Z}S(I-\varepsilon D_1)\widehat{A}^{\|\widehat{D}}(I-\varepsilon D_2)T\\
		&=&\widehat{Z} S(I-\varepsilon D_1)\widehat{D} \widehat{L}(I-\varepsilon D_2)T\\
		&=&\widehat{Z} S(I-\varepsilon D_1)(I+\varepsilon D_1)GH(I+\varepsilon D_2)\widehat{L}(I-\varepsilon D_2)T\\
		&=&\widehat{Z}H(I+\varepsilon D_2)\widehat{L}(I-\varepsilon D_2)T\\
		&=&H(I+\varepsilon D_2)\widehat{A}(I+\varepsilon D_1)GH(I+\varepsilon D_2)\widehat{L}(I-\varepsilon D_2)T\\
		&=&H(I+\varepsilon D_2)\widehat{A}\widehat{D}\widehat{L}(I-\varepsilon D_2)T\\
		&=&H(I+\varepsilon D_2)\widehat{A}\widehat{A}^{\|\widehat{D}}(I-\varepsilon D_2)T\\
		&=&S(I-\varepsilon D_1)(I+\varepsilon D_1)GH(I+\varepsilon D_2)\widehat{A}\widehat{A}^{\|\widehat{D}}(I-\varepsilon D_2)T\\
		&=&S(I-\varepsilon D_1)\widehat{D}\widehat{A}\widehat{A}^{\|\widehat{D}}(I-\varepsilon D_2)T\\
		&=&S(I-\varepsilon D_1)\widehat{D}(I-\varepsilon D_2)T\\
		&=&S(I-\varepsilon D_1)(I+\varepsilon D_1)GH(I+\varepsilon D_2)(I-\varepsilon D_2)T=I.
	\end{eqnarray*}
	Analogously, we have $S(I-\varepsilon D_1)\widehat{A}^{\|\widehat{D}}(I-\varepsilon D_2)T\widehat{Z}=I$, which completes the proof.
	
	(3)$\Rightarrow$(1): Take  $\widehat{X}=(I+\varepsilon D_1)G\widehat{Z}^{-1}H(I+\varepsilon D_2).$ Since $H(I+\varepsilon D_2)(I-\varepsilon D_2)T=I$, we have
	\begin{eqnarray*}
		\widehat{X}&=&(I+\varepsilon D_1)G\widehat{Z}^{-1}H(I+\varepsilon D_2)\\
		&=&(I+\varepsilon D_1)G(H(I+\varepsilon D_2)(I-\varepsilon D_2)T)\widehat{Z}^{-1}H(I+\varepsilon D_2)\\
		&=&\widehat{D}(I-\varepsilon D_2)T\widehat{Z}^{-1}H(I+\varepsilon D_2).
	\end{eqnarray*} That is, there exists $\widehat{Y_1}\in \mathbb{D}^{m\times m}$ such that $\widehat{X}=\widehat{D}\widehat{Y_1}$. Then
	\begin{eqnarray*}
		\widehat{D}\widehat{A}\widehat{X}&=&(I+\varepsilon D_1)GH(I+\varepsilon D_2)\widehat{A}(I+\varepsilon D_1)G\widehat{Z}^{-1}H(I+\varepsilon D_2)\\
		&=&(I+\varepsilon D_1)G\widehat{Z}\widehat{Z}^{-1}H(I+\varepsilon D_2)=\widehat{D}.
	\end{eqnarray*}
	A similar argument shows that $\widehat{D}=\widehat{X}\widehat{A}\widehat{D}$ and that $\widehat{X}=\widehat{Y_2}\widehat{D}$ for some $\widehat{Y_2}\in \mathbb{D}^{n\times n}$. Hence, $\widehat{A}^{\|\widehat{D}}$ exists.
\end{proof}

Taking $\widehat{D}=\widehat{A}^{\mathrm T}$ in Proposition \ref{bc4.6}, we obtain the next result.

\begin{corollary} Let $\widehat{A}=A+\varepsilon A_0\in \mathbb{D}^{m\times n}$ and $A=GH$ be the full rank decomposition. If $\widehat{A}=(I+\varepsilon A_1)A(I+\varepsilon A_2)$ for some real matrices $A_1,A_2$, then
	\begin{eqnarray*}
		\widehat{A}^{\dagger}=&&(I+\varepsilon A^{\mathrm T}_2)H^{\mathrm T}(H H^{\mathrm T}+\varepsilon HA_2H^{\mathrm T}+\varepsilon HA_2^{\mathrm T}H^{\mathrm T})^{-1} \bullet \\
		&&(G^{\mathrm T}G +\varepsilon G^{\mathrm T}A_1^{\mathrm T}G+\varepsilon G^{\mathrm T}A_1G)^{-1}G^{\mathrm T}(I+\varepsilon A^{\mathrm T}_1).
	\end{eqnarray*}
\end{corollary}

Taking $\widehat{D}=\widehat{A}$ in Proposition \ref{bc4.6} yields  the following result. The equivalence of conditions (1), (2) and (3) in this result was also established in \cite{WangGao2023,Zhongjin2022}.

\begin{corollary} Let $\widehat{A}=A+\varepsilon A_0\in \mathbb{D}^{n\times n}$ and $A=GH$ be the full rank decomposition. The following statements are equivalent:
	\begin{itemize}
		\item [{\rm (1)}] $\widehat{A}^{\#}$ exists;
		\item [{\rm (2)}] $\widehat{A}^{\scriptsize\textcircled{\tiny \#}}$ exists;
		\item [{\rm (3)}] $A^{\#}$ exists with $(1-AA^-)A_0(1-A^-A)=0$ for some $A^-\in A\{1\}$;
		\item [{\rm (4)}] $\widehat{A}=(I+\varepsilon A_1)A(I+\varepsilon A_2)$ for some real matrices $A_1,A_2$, and $\widehat{Z}= H(I+\varepsilon A_2)(I+\varepsilon A_1)G$  is invertible.
	\end{itemize}
	In this case,
	$\widehat{A}^{\#}=(I+\varepsilon A_1)G\widehat{Z}^{-2}H(I+\varepsilon A_2)$,\\
	$\widehat{A}^{\scriptsize\textcircled{\tiny \#}}=(I+\varepsilon A_1)G\widehat{Z}^{-1}(G^{\mathrm T}G +\varepsilon G^{\mathrm T}A_1^{\mathrm T}G+\varepsilon G^{\mathrm T}A_1G)^{-1}G^{\mathrm T}(I+\varepsilon A^{\mathrm T}_1).$
\end{corollary}
\begin{proof} The equivalence between (1) and (4) is clear by Proposition \ref{bc4.6}.  Taking $\widehat{D}=\widehat{A}$ in Proposition \ref{bc4.6} yields   $\widehat{A}^{\#}=(I+\varepsilon A_1)G\widehat{Z}^{-2}H(I+\varepsilon A_2)$. It follows from \cite[Theorem 3.7]{WangGao2023} that $\widehat{A}^{\scriptsize\textcircled{\tiny \#}}=\widehat{A}^{\#}\widehat{A}\widehat{A}^{\dagger}$, and consequently  $$\widehat{A}^{\scriptsize\textcircled{\tiny \#}}=(I+\varepsilon A_1)G\widehat{Z}^{-1}(G^{\mathrm T}G +\varepsilon G^{\mathrm T}A_1^{\mathrm T}G+\varepsilon G^{\mathrm T}A_1G)^{-1}G^{\mathrm T}(I+\varepsilon A^{\mathrm T}_1).$$
\end{proof}

We now consider the Drazin inverse of a dual square matrix. For $\widehat{A}=A+\varepsilon A_0\in \mathbb{D}^{n\times n}$, we obtain $\widehat{A}^{s}=A^s+\varepsilon({\sum\limits_{i=0}^{s-1} A^iA_0A^{s-1-i}})$ by an inductive argument. If ${\rm i}(A)=k$, then $\widehat{A}^{2k}=A^{2k}+\varepsilon({\sum\limits_{i=0}^{2k-1} A^iA_0A^{2k-1-i}})$ is regular due to $(1-AA^D)({\sum\limits_{i=1}^{2k-1} A^iA_0A^{2k-1-i}})(1-AA^D)=0$. By Corollary \ref{bc3.6}, we get that $\widehat{A}^D$ exists with ${\rm i}(\widehat{A})\leq 2k$, meaning that $\widehat{A}^{\|\widehat{A}^{2k}}$ exists. It is worth noting that Zhong and Zhang \cite{Zhongjin2023} also investigated a similar problem, namely the existence of $\widehat{A}^{\|\widehat{A}^{k}}$. The above results thus offer potential methods for future investigation of  the DMP inverse \cite{MT} of a dual matrix.

The idempotence of generalized inverses is a topic of ongoing interest in this field (cf. \cite{Baksalary2020,Bernstein2018,MosicZhang2024}). Using Corollaries \ref{bc3.9} and \ref{bc3.10}, one may investigate the idempotence of generalized inverses of dual matrices; we omit the details here. Additionally, we do not present the results concerning the absorption law for $\{2\}$-inverses of dual matrices.

Finally, we establish the following two results regarding  the clean property of $\widehat{A}=A+\varepsilon A_0\in \mathbb{D}^{n\times n}$.

\begin{itemize}
	\item [{\rm (1)}] $\widehat{A}$ is strongly clean, since $\widehat{A}^D$ exists (cf. \cite[Theorem 1]{Nicholson1999}).
	\item [{\rm (2)}] $\widehat{A}$ is special clean if and only if $\widehat{A}$ is regular, since each square real matrix is special clean.
\end{itemize}

\subsection*{Acknowledgment}
The authors would like to  thank Dr. Shuxian Xu for meaningful suggestions.The first   author is supported by the National Natural Science Foundation of China (Grant No. 12171083). The second author has been partially supported by Universidad Nacional de La Pampa (Grant Resol. 172/2024), by Grant PGI 24/ZL22, Departamento de Matem\'atica, Universidad Nacional del Sur (UNS), Argentina; by Ministerio de Ciencia, Innovaci\'on y Universidades of Spain (Grant Redes de Investigaci\'on, MICINN-RED2022-134176-T), and by Universidad Nacional de R\'{\i}o Cuarto (Grant Resol. RR 449/2024).

\end{document}